\documentclass[11pt,reqno]{amsart}
\usepackage{amsthm}
\usepackage{amssymb}
\usepackage{graphics}
\usepackage{tikz}
\usetikzlibrary{shapes,backgrounds,calc}
\usepackage{latexsym}
\usepackage{multicol}
\usepackage{verbatim,enumerate}
\usepackage{accents}
\usepackage{cite}

\usepackage[colorlinks=true, linkcolor=blue, citecolor=blue, urlcolor=black]{hyperref}
\usepackage{hyperref}
\usepackage{amsmath, amscd,url}

\advance\textwidth by 1.3in \advance\oddsidemargin by -.6in \advance\evensidemargin by -.6in
\DeclareMathOperator{\htt}{ht}

\theoremstyle{definition}

\newtheorem{thm}{Theorem}[section]
\newtheorem{cor}[thm]{Corollary}
\newtheorem{prop}[thm]{Proposition}
\newtheorem{lem}[thm]{Lemma}

\theoremstyle{Definition}
\newtheorem*{defn}{Definition}

\theoremstyle{definition}

\newtheorem*{rem}{Remark}

\newenvironment{pf}{\proof}{\endproof}
\newcounter{cnt}
 \makeatletter
\def\mydggeometry{\makeatletter\dg@YGRID=1\dg@XGRID=20\unitlength=0.003pt\makeatother}
\makeatother \theoremstyle{remark}

\numberwithin{equation}{section}
\let\bwdg\bigwedge
\def\bigwedge{{\textstyle\bwdg}}

\newcommand{\thmref}[1]{Theorem~\ref{#1}}

\newcommand{\wt}{\operatorname{wt}}

\newcommand{\nc}{\newcommand}
\newcommand{\rnc}{\renewcommand}

\nc{\cal}{\mathcal} \nc{\goth}{\mathfrak} \rnc{\bold}{\mathbf}

\newcommand{\supp}{\operatorname{supp}}
\newcommand{\chrmult}{\operatorname{chrmult}}

\nc\bomega{{\mbox{\boldmath $\omega$}}} \nc\bpsi{{\mbox{\boldmath $\Psi$}}}
 \nc\balpha{{\mbox{\boldmath $\alpha$}}}
 \nc\bpi{{\mbox{\boldmath $\pi$}}}
 \nc\bvpi{{\mbox{\boldmath $\varpi$}}}
\nc\chara{\operatorname{ch}}

  \nc\bxi{{\mbox{\boldmath $\xi$}}}
\nc\bmu{{\mbox{\boldmath $\mu$}}} \nc\bcN{{\mbox{\boldmath $\cal{N}$}}} \nc\bcm{{\mbox{\boldmath $\cal{M}$}}} \nc\blambda{{\mbox{\boldmath
$\lambda$}}}\nc\bnu{{\mbox{\boldmath $\nu$}}}

\newcommand{\lie}[1]{\mathfrak{#1}}

\makeatletter
\def\section{\def\@secnumfont{\mdseries}\@startsection{section}{1}%
  \z@{.7\linespacing\@plus\linespacing}{.5\linespacing}%
  {\normalfont\scshape\centering}}
\def\subsection{\def\@secnumfont{\bfseries}\@startsection{subsection}{2}%
  {\parindent}{.5\linespacing\@plus.7\linespacing}{-.5em}%
  {\normalfont\bfseries}}
\makeatother

 \nc{\Hom}{\operatorname{Hom}}
  \nc{\mode}{\operatorname{mod}}
\nc{\End}{\operatorname{End}} \nc{\wh}[1]{\widehat{#1}} \nc{\Ext}{\operatorname{Ext}} \nc{\ch}{\text{ch}} \nc{\ev}{\operatorname{ev}}
\nc{\Ob}{\operatorname{Ob}} \nc{\soc}{\operatorname{soc}} \nc{\rad}{\operatorname{rad}} \nc{\head}{\operatorname{head}}

\def\mult{\operatorname{mult}}
\def\st{\operatorname{st}}

 \nc{\Cal}{\cal} \nc{\Xp}[1]{X^+(#1)} \nc{\Xm}[1]{X^-(#1)}
\nc{\on}{\operatorname} \nc{\Z}{{\bold Z}} \nc{\J}{{\cal J}}  \nc{\Q}{{\bold Q}}

\nc{\N}{{\bold N}}  \nc\boa{\bold a} \nc\bob{\bold b} \nc\boc{\bold c} \nc\bod{\bold d} \nc\boe{\bold e} \nc\bof{\bold f} \nc\bog{\bold g}
\nc\boh{\bold h} \nc\boi{\bold i} \nc\boj{\bold j} \nc\bok{\bold k} \nc\bol{\bold l} \nc\bom{\bold m} \nc\bon{\mathbb n} \nc\boo{\bold o}
\nc\bop{\bold p} \nc\boq{\bold q} \nc\bor{\bold r} \nc\bos{\bold s} \nc\boT{\bold t} \nc\boF{\bold F} \nc\bou{\bold u} \nc\bov{\bold v}
\nc\bow{\bold w} \nc\boz{\bold z}\nc\ba{\bold A} \nc\bb{\bold B} \nc\bc{\mathbb C} \nc\bd{\bold D} \nc\be{\bold E} \nc\bg{\bold
G} \nc\bh{\bold H} \nc\bi{\bold I} \nc\bj{\bold J} \nc\bk{\bold K} \nc\bl{\bold L} \nc\bm{\bold M} \nc\bn{\mathbb N} \nc\bo{\bold O} \nc\bp{\bold
P} \nc\bq{\bold Q} \nc\br{\bold R} \nc\bs{\bold S} \nc\bt{\bold T} \nc\bu{\bold U} \nc\bv{\bold V} \nc\bw{\bold W} \nc\bz{\mathbb Z} \nc\bx{\bold
x} \nc\KR{\bold{KR}} \nc\rk{\bold{rk}} \nc\het{\text{ht }}

\nc\toa{\tilde a} \nc\tob{\tilde b} \nc\toc{\tilde c} \nc\tod{\tilde d} \nc\toe{\tilde e} \nc\tof{\tilde f} \nc\tog{\tilde g} \nc\toh{\tilde h}
\nc\toi{\tilde i} \nc\toj{\tilde j} \nc\tok{\tilde k} \nc\tol{\tilde l} \nc\tom{\tilde m} \nc\ton{\tilde n} \nc\too{\tilde o} \nc\toq{\tilde q}
\nc\tor{\tilde r} \nc\tos{\tilde s} \nc\toT{\tilde t} \nc\tou{\tilde u} \nc\tov{\tilde v} \nc\tow{\tilde w} \nc\toz{\tilde z} \nc\woi{w_{\omega_i}}

\begin{document}
\setcounter{section}{0}
\setcounter{tocdepth}{1}

%%%%%%%%%%%%%%%%%%%%%%%%%%%%%%%%%%%%%%%%%%%%

\title{Generalized Chromatic polynomials of graphs from heaps of pieces}

\author{G. Arunkumar}
\address{Indian Institute of Science Education and Research, Mohali, India.}
\email{arun.maths123@gmail.com \\ gakumar@iisermohali.ac.in.}

\thanks{Part of this work was done when the author was a graduate student at the Institute of Mathematical Sciences, Chennai, India; and the author acknowledges the partial support under the DST Swarnajayanti fellowship DST/SJF/MSA-02/2014-15 of Amritanshu Prasad. The author is grateful to Indian Institute of Science Education and Research, Mohali for the institute post doctoral fellowship.}

\subjclass [2010]{05C15, 05C31, 05E15, 17B01, 17B67}
\keywords{Acyclic orientations, Chromatic polynomials, Free partially commutative Lie algebras, Heaps of pieces, Lyndon heaps, Pyramid proportionality lemma, Reciprocity theorem}

\begin{abstract}
	Let $G$ be a simple graph and let $\mathcal{L}(G)$ be the free partially commutative Lie algebra associated to $G$. In this paper, using heaps of pieces, we prove an expression for the generalized $\bold k$-chromatic polynomial of $G$ in terms of dimensions of the grade spaces of $\mathcal{L}(G)$. This will give us a new interpretation for the chromatic polynomials in terms of multilinear heaps and Lyndon length. The classical results of Stanley, and Greene and Zaslavsky regarding the acyclic orientations of $G$ are obtained as corollaries.  A heap with a unique minimal piece is said to be a pyramid and our main theorem is proved using the properties of pyramids. For this reason, we prove two important properties of pyramids namely the pyramid proportionality lemma, and the pyramid and Lyndon heap lemma. In the last section, we will introduce the $(m,\lambda)$-labelling on acyclic orientations which is similar to Stanley's $\lambda$-compatible pairs. As an application of our main theorem, using this $(m,\lambda)$-labelled acyclic orientations, we will prove the reciprocity theorem for the derivatives of the chromatic polynomials.
	%More precisely, we will introduce a new $(m,\lambda)$-labelling on acyclic orientations similar to Stanley's $\lambda$-compatible pairs and we will prove that the number of such labelling is counted by the $m$-th derivative of the chromatic polynomial evaluated at the negative integers $\lambda$. 
	 %When $m=0$, this gives a new combinatorial model counted by the chromatic polynomial evaluated at negative integers.
\end{abstract}

\maketitle

\section{Introduction}\label{intro}
We denote the set of complex numbers by $\bc$ and, respectively, the set of integers, non-negative integers, and positive integers by $\bz$, $\bz_+$, and $\bn$.  Let $G$ be a finite simple graph with a totally ordered vertex set $I$ and edge set $E(G)$. Assume that $I = \{\alpha_1,\alpha_2,\dots,\alpha_n\}$ and $\bold k = (k_1,k_2,\dots,k_n) \in \mathbb{Z}^{n}_{+}$. When there is no confusion we will denote the vertex $\alpha_i$ simply by $i$. 
% The chromatic polynomial of $G$, denoted by $\pi_\bold 1^G(q)$, was introduced by George David Birkhoff in the year 1912 \cite{Birkhoff}; since then it has been studied in various aspects and become an important graph invariant. 
In this paper, we are interested in the generalized $\bold k$-chromatic polynomial of $G$ \cite{S98,HMK04,akv} which is defined as follows.  A vertex multicoloring of $G$ associated to $\bold k$ is an assignment of colors to the vertices of $G$ in which each vertex $\alpha_i\in I$ receives exactly $k_i$ colors such that the adjacent vertices receive disjoint colors. The generalized $\bold k$-chromatic polynomial ($\bold k$-chromatic polynomial in short)  counts the number of distinct proper vertex multicolorings of $G$ using $q$ colors and is denoted by $\pi^G_{\bold k}(q)$.  Note that, if $\bold k = \bold 1 := (1,1,\dots,1)$ then the multicoloring corresponds to the classical vertex coloring of $G$. There is a close relationship between the ordinary chromatic polynomials and the $\bold k$-chromatic polynomials. We have
\begin{equation}\label{connection}
	\pi_\bold k^G(q)=\frac{1}{\bold k!}\pi_{\bold 1}^{G(\bold k)}(q)
\end{equation}
 where $\pi_{\bold 1}^{G(\bold k)}(q)  \text{ is the chromatic polynomial of the graph } G(\bold k)$ and $\bold k!:=\prod_{i\in I}k_i!$. The definition of the graph $G(\bold k)$ is given in Section \ref{ppl1}.
 	It is well-known that the coefficients of $\pi_{\bold k}^G(q)$ are alternate in sign \cite[Theorem 10]{MR0224505}. We are interested in relating these coefficients to the cardinality of some well-known combinatorial objects. For this reason, we  define a positive variant of $\pi^G_{\bold k}(q)$  as follows: $$\widetilde{\pi}^G_{\bold k}(q) = (-1)^{\htt \bold k}\,\pi^G_{\bold k}(-q),$$
 where $\htt \bold k := \sum_{i = 1}^n k_i$. We note that the coefficients of $\widetilde{\pi}^G_{\bold k}(q)$ are non-negative. 
%\begin{defn}\label{ing}
	Let $L_{G}(\bold k)$ be the bond lattice of $G$ of weight $\bold k$, which consists of 
	$\mathbf{J}=(J_1,\dots,J_k)$  satisfying the following properties:
	\begin{enumerate}
		\item[(i)] $\bold J$ is a multiset, i.e., we allow $J_i=J_j$ for $i\neq j$,
		
		\item[(ii)] each $J_i$ is a multiset and the subgraph spanned by the underlying set of $J_i$ is connected subgraph of $G$ for each $1\le i\le k$,
		
		\item[(iii)] the disjoint union $J_1\dot{\cup} \cdots \dot{\cup} J_k=\{\alpha_i,\dots, \alpha_i: \alpha_i \text{ occurs $k_i$ times for }i\in I\}$ and
		
		\item[(iv)] the minimum elements of parts of $\bold J$ are non-increasing, i.e., $\min J_1 \ge  \dots \ge \min J_k$.
	\end{enumerate}
We define $\wt J_i$ to be $\bold {m}_i \in \mathbb{Z}_+^n$ which counts the number of times an element of $I$ occurs in $J_i$ and $\wt \bold J := \sum_{i =1}^k \wt J_i$.  Let  $\lie g$ be a Kac-Moody Lie algebra  with the associated graph $G$ \cite{VV15}. Assume $\bold k = \bold 1$, then the entries of $\bold {m}_i$ are either zero or one.  Since each $J_i$ is a connected subset of $G$, $\bold {m}_i$ represents a positive root in $\lie g$ \cite[Proposition 4]{VV15} and hence the root space $\lie g_{{\bold m}_i} \ne 0$. The dimension of the root space $\lie g_{{\bold m}_i}$ is denoted by $\mult \bold {m}_i$. The definition of multiplicity can be extended to an element $\bold J$ of $L_G(\bold 1)$ by defining $\mult \bold J := \prod_{i=1}^k \mult \bold{m}_i$.
%\end{defn}
%Consider the bond lattice $L_G(\bold 1)$ of $G$, consists of set partitions of $I$ in which each part induces a connected subgraph of $G$.  
Given this, the following result expresses the chromatic polynomial in terms of bond lattice and the root multiplicities of $\lie g$  \cite[Theorem 1.1]{VV15}.  
\medskip
\begin{thm}\label{kac}
	Let  $\lie g$ be a Kac-Moody Lie algebra  with the associated graph $G$, then
	\begin{equation}\label{kaceq}
	\widetilde{\pi}_\bold 1^G(q) = \sum\limits_{\bold J \in L_G(\bold 1)} \mult \bold J\, q^{|\bold J|}
	\end{equation}
	where $|\bold J|$ denotes the number of parts in $\bold J$.
\end{thm}
 We want to prove this theorem for free partially commutative Lie algebra associated to the graph $G$ using heaps of pieces. Also, we want to extend this result to the $\bold k$-chromatic polynomials. Let $\mathcal{FL}(I)$ be the free Lie algebra on the vertex set $I$ and let $J$ be the ideal in $\mathcal{FL}(I)$ generated by the relations $\{[i,j] : \{i,j\} \notin E(G) \}$. The quotient algebra $\frac{\mathcal{FL}(I)}{J}$, denoted by $\mathcal{L}(G)$, is the free partially commutative Lie algebra associated to the graph $G$. It is well-known that $\mathcal{FL}(I)$ and hence $\mathcal{L}(G)$ is graded by $\mathbb{Z}_{+}^n$.
For $\bold k \in \mathbb{Z}_{+}^n$, we define, $\mult \bold k = \dim \mathcal{L}_{\bold k}(G)$.   Similar to the Kac-Moody algebra case, for $\bold J \in L_G(\bold 1)$ (with the notation as above) we define $\mult \bold J := \prod_{i=1}^k \mult \bold{m}_i$ product of dimensions of grade spaces of $\mathcal{L}(G)$. 
The absolute value of the linear coefficient of the chromatic polynomial, denoted by $\widetilde{\pi}_{\bold 1}^G(q)[q]$, is known as the chromatic discriminant of the graph $G$ \cite{MR1861053,a2}. 
Theorem \ref{kac} implies that, 
	$\mult \bold {m}_i = \widetilde{\pi}^G_{\bold {m}_i}(q)[q]$.  Hence $\mult \bold J =  \prod_{i=1}^k (\widetilde{\pi}^G_{\bold {m}_i}(q)[q])$   for $\bold J \in L_G(\bold 1)$.	
		When we extend Theorem \ref{kac} to the case of $\bold k$-chromatic polynomials, the bond lattice of weight $\bold k$ shows up naturally. Consider $\bold J = (J_1,\dots,J_1,\dots,J_r,\dots,J_r) \in L_G(\bold k)$ where each $J_i$ has weight $\bold {m}_i$ and occurs $m_i$ times.
		Since $\bold k$ is arbitrary, $J_i$ can be a multiset and hence the entries of $\bold {m}_i$ need not be zero or one. This shows that, because of Equation \eqref{connection}, the linear coefficient of $\bold {m}_i$-chromatic polynomial need not be an integer. Hence, $\mult \bold {m}_i$ need not be equal to $\widetilde{\pi}^G_{\bold {m}_i}(q)[q]$. We will see that, by Corollary \ref{recursionmult}, $\widetilde{\pi}^G_{\bold {m}_i}(q)[q]$ is equal to $\sum\limits_{l_i | \bold {m}_i}\frac{\mult (\frac{\bold {m}_i}{l_i})}{l_i}$. This suggests us to define the chromatic multiplicity of $\bold {m}_i$ as $\chrmult\, \bold {m}_i := \sum\limits_{l_i | \bold {m}_i}\frac{\mult (\frac{\bold {m}_i}{l_i})}{l_i}$ and
	$$\chrmult\, \bold J := \prod\limits_{i =1}^r \frac{1}{m_i!} \Big(\chrmult\, \bold {m}_i\Big)^{m_i}.$$ 
We have $\mult \bold J = \chrmult \bold J$ for $\bold J \in L_G(\bold 1)$ and also $\chrmult \bold J$ is equal to $ \prod\limits_{i =1}^r \frac{1}{m_i!} \Big(\widetilde{\pi}^G_{\bold {m}_i}(q)[q]\Big)^{m_i}$.  With this extended definition of multiplicity,
we prove the following theorem for free partially commutative Lie algebra $\mathcal{L}(G)$ using heaps of pieces.
\medskip
 \begin{thm} \label{mainthm}
 	With the notation as above we have
	\begin{equation}\label{heapexp}
	\widetilde{\pi}^G_{\bold k}(q) \,=\,\sum\limits_{\bold J \in L_G(\bold k)} \chrmult \bold J \,\,q^{|\bold J|}, 
	\end{equation}
	where $|\bold J|$ denotes the number of parts in $\bold J$.
\end{thm}
We observe that the multiplicities appearing in Theorem \ref{kac} are of Kac-Moody Lie algebras whereas in the above theorem the multiplicities are of free partially commutative Lie algebras.  This gives the following connection between free partially commutative Lie algebras and Kac-Moody Lie algebras.
\begin{cor}\label{kacfree}
	Let $\lie g$ be a Kac-Moody Lie algebra with the associated graph $G$. Let $\bold k \in \mathbb{Z}_+^n$ be such that its entries are either zero or one, then $\dim \lie g_{\bold k} = \dim \mathcal{L}_{\bold k}(G)$.
\end{cor}	
\begin{rem}
 Let $\lie g$ be a Kac-Moody Lie algebra with the associated graph $G$ and let $\lie n^+$ be its positive part, then the defining relations of $\lie n^{+}$ are given by the Serre relations \cite[Theorem 1.2 and Theorem 9.11]{kac}. We observe that if all the entries of $\bold k$ are either zero or one then the Serre relations reduce to the commutation relations of $\mathcal{L}(G)$. This explains the previous corollary.
\end{rem}

In \cite{VV15}, the denominator identity for Kac-Moody Lie algebras is the main tool used in the proof of Theorem \ref{kac}.
We prove Theorem \ref{mainthm} using the fundamental lemmas of Viennot for heaps of pieces, the pyramid proportionality lemma, and the pyramid and Lyndon heap lemma. In this sense, our proof is purely heap theoretic. Our motivation to use heaps of pieces came from the works of Xavier Viennot and Lalonde: In \cite{MR11108}, the combinatorial theory of heaps of pieces was introduced by Xavier Viennot where he gave the applications of heaps of pieces to a wide range of areas like directed animals, polyominoes, Motzkin paths and orthogonal polynomials, Rogers-Ramanujan identities, fully commutative elements in Coxeter groups, Bessel functions and Lorentzian quantum gravity. 
The applications of heaps of pieces to the representation theory of complex simple Lie algebras can be seen in \cite{Green} where the combinatorial aspects of minuscule representation are studied.  
%The fundamental lemmas of Xavier Viennot for heaps of pieces play an important role in our proof of Theorem \ref{mainthm}. 
Using the notion Lyndon length of heaps we will prove the following expression for chromatic polynomial in terms of multilinear heaps and Lyndon length. See,  Proposition \ref{lyndonfact} for the precise definition of Lyndon length.

  \begin{cor}\label{chmhs}
  	Let $\mathcal{H}_{\bold 1}(I,\zeta)$ be the set of heaps of weight $\bold 1$, then
	\begin{equation}
	\widetilde{\pi}_\bold 1^G(q) = \sum\limits_{k \ge 1} | \{E \in \mathcal{H}_{\bold 1}(I,\zeta) : ll(E) = k \}|  q^k = \sum\limits_{E \in \mathcal{H}_{\bold 1}(I,\zeta)} q^{ll(E)}
	\end{equation}
	where $ll(E)$ denotes the Lyndon length of $E$.
\end{cor}
We extend the definition of Lyndon length to acyclic orientations of $G$ by using a natural bijection between multilinear heaps and acyclic orientations. This will give us the following expression for chromatic polynomial in terms of acyclic orientations.
\begin{cor}\label{acycliccor}
	Let $\mathcal{O}(G)$ be the set of all acyclic orientations of $G$, then
	\begin{equation}\label{1acyclicexp}
	\widetilde{\pi}^G_{\bold 1}(q) \,= \sum\limits_{k \ge 0} \,|\{\mathcal{O} \in \mathcal{O}(G) : ll(\mathcal{O}) = k\}| \, q^k = \sum\limits_{\mathcal{O} \,\in\, \mathcal{O}(G)} q^{ll(\mathcal{O})}.
	\end{equation}
\end{cor}
%We get the classical results of Stanley \cite{S73} and Greene and Zaslavsky \cite{GZ83,MR1778205} as immediate consequences of these corollaries. 

 In the final section, we will prove results regarding Stanley's chromatic polynomial reciprocity theorem.   The following results of Stanley \cite{S73}, and Greene and Zaslavsky \cite{GZ83} respectively are classical and follow immediately from Corollary \ref{acycliccor}. In \cite{deb19}, these results are proved using the method of involution on heaps. In this section, we will denote the chromatic polynomial by $\chi_G(q)$ for notational convenience.
 \begin{thm}\label{0stanley} 
 	With the notation as above, we have
 	$$ \widetilde{\chi}_G(1) = \text{ the number of acyclic orientations of $G$}.$$
 \end{thm}
 \begin{thm} 
 	Fix $i \in I$.
 	$$\widetilde{\chi}_G[q] = \text{ The number of acyclic orientations of $G$ with a unique source at $\alpha_i$}.$$
 \end{thm}
 In \cite{S73}, Theorem \ref{0stanley} is stated in the following more general context of $\lambda$-compatible pairs. 
 Let $\mathcal{O}$ be an acyclic orientation of $G$ and for a positive integer $\lambda$ define $[\lambda] := \{1,2,\dots,\lambda\}$. For a map $\sigma: I \to [\lambda]$, we say $(\sigma, \mathcal{O})$ is a $\lambda$-compatible pair if for each directed edge $i\rightarrow j$ in $\mathcal{O}$
 we have $\sigma(i)\ge \sigma(j)$. 
 \begin{thm}\label{1stanley}
 	Let $\lambda$ be a positive integer, then
 	$$ \widetilde{\chi}_G(\lambda) = \text{ the number of $\lambda$-compatible pairs of $G$}.$$
 \end{thm}
 This theorem is known as Stanley's reciprocity theorem for chromatic polynomials. We observe that the $\lambda$-compatible pairs are acyclic orientations which are labelled subject to a compatibility condition coming from the underlying poset structure. We will define a similar labelling on acyclic orientations in which the compatibility condition is coming from its Lyndon factorization. For a positive integer $\lambda$ and a non-negative integer $m$, we will introduce the $(m,\lambda)$-labelling on acyclic orientations. We will prove the following reciprocity theorem for the derivatives of chromatic polynomials.
 \begin{thm}\label{(m,l)}
 	Let $\lambda$ be a positive integer and let $0 \le m \le n$ be an integer then
 	$$ \widetilde{\chi}_G^{(m)}(\lambda) = \text{ the number of $(m,\lambda)$-labelled acyclic orientations of $G$}.$$
 \end{thm}
 
This theorem is analogous to Theorem \ref{1stanley} for the case of the derivatives $\widetilde{\chi}_G^{(m)}(q)$ and $(m,\lambda)$-labelled acyclic orientations. When $m=0$, we call $(m,\lambda)$-labelling simply $\lambda$-labelling and this combinatorial model is counted by the chromatic polynomial evaluated at the negative integers:
 \begin{cor}
	$$ \widetilde{\chi}_G(\lambda) = \text{ the number of $\lambda$-labelled acyclic orientations of $G$}$$
\end{cor}

 It is well-known that the Lyndon words index a basis for free Lie algebras. Lalonde studied Lyndon heaps which are free partially commutative analogous of Lyndon words.
 In \cite[Section 4]{MR1235180}, he proved the following result.
 \begin{thm}\label{lalonde}
 	Let $\mathcal{L}_{\bold k}(G)$ be the $\bold k$ grade space of the free partially commutative Lie algebra $\mathcal{L}(G)$, then
 	\begin{equation}
 	\dim \mathcal{L}_{\bold k}(G) = \text{ number of Lyndon heaps of weight $\bold k$}.
 	\end{equation}
 \end{thm} 
 This theorem shows the natural connection between heaps of pieces and the free partially commutative Lie algebras.  In this direction, we want to address the following question: Every Lyndon heap is a pyramid and hence, because of Theorem \ref{lalonde}, one naturally asks for the relation between the number of pyramids and the dimensions of grade spaces of free partially commutative Lie algebras. We answer this question in the following proposition whose proof is given in Section \ref{free}. 
 \medskip
 \begin{prop}\label{pl}
 	Let $\mathcal{P}_{\bold k}(I,\zeta)$ be the set of pyramids of weight $\bold k$, then
$$|\mathcal{P}_{\bold k}(I,\zeta)| = \sum_{l | \bold k} \frac{\htt \bold k}{l} \dim \mathcal{L}_{\frac{\bold k}{l}}(G).$$
In particular, if $\bold k$ is relatively prime then
$|\mathcal{P}_{\bold k}(I,\zeta)| = (\htt \bold k) \dim \mathcal{L}_{\bold k}(G).$
 \end{prop}
 
{\em Acknowledgements. The author is grateful to K. N. Raghavan, Sankaran Viswanath, Tanusree Khandai, R. Venkatesh and Xavier Viennot for many helpful discussions and constant support.}

 \section{Pyramid proportionality lemma}\label{ppl}
\subsection{Heaps monoid}\label{basic}
Let $G$ be a graph with a totally ordered vertex set $I = \{\alpha_1,\dots,\alpha_n\}$. 
The \emph{free monoid} on $I$, denoted by $\mathcal{M}(I)$, is totally ordered by the lexicographic order induced from the total order on $I$. We say two elements $i$ and $j$ of $I$ commute if $\{i,j\} \notin E(G)$. We use this commutation relation on $I$ to define an equivalence relation $\eta$ on $\mathcal{M}(I)$: Two words $w_{1}$ and $w_{2}$ in $\mathcal{M}(I)$ are related by $\eta$ if $w_{2}$ is obtained from $w_{1}$ by a sequence of interchanges of adjacent commuting alphabets. The \emph{free partially commutative monoid} associated with $G$, denoted by $\mathcal{M}(I,\eta)$, is defined to be the set of all equivalence classes $\frac{\mathcal{M}(I)}{\eta}$. Note that, $\mathcal{M}(I,\eta)$ has a natural monoid structure induced from the monoid structure on $\mathcal{M}(I)$.

Let $\zeta$ be the concurrency relation complement to the commuting relation $\eta$. A \emph{pre-heap} $E$ over $(I,\zeta)$ is a finite subset of $I \times \{0,1,2, \dots \}$ satisfying, if $(\alpha_1,i),(\alpha_2,j) \in E$ with $\alpha_1 \,\zeta\, \alpha_2$, then $i \ne j$. Each element $(\alpha,i)$ of $E$ is called a basic piece.  If $(\alpha,i) \in E$, we write $\pi(\alpha,i) = \alpha $ (the position of the piece $(\alpha,i)$) and $h(\alpha,i) = i$ (the level of the piece $(\alpha,i)$). A basic piece will be simply denoted by $\alpha$ when we don't need to emphasize on the level. The set $\pi(E)$ is defined to be the set of all positions occupied by the pieces of $E$.
A pre-heap $E$ defines a partial order $\le_E$ by taking the transitive closure of the relation : 
$(\alpha_1,i) \le_E (\alpha_2,j)$ if $\alpha_1 \zeta \alpha_2$ and $i < j$. 
We say that two heaps $E$ and $F$ are \emph{isomorphic} if there exists a position preserving order isomorphism $\phi$ between $(E,\le_E) $ and $(F,\le_F)$. 
A \emph{heap} $E$ over $(I,\zeta)$ is a pre-heap over $(I,\zeta)$ such that: if $(\alpha,i) \in E$ with $i > 0$ then there exists $(\beta,i-1) \in E$ such that $\alpha \zeta \beta$.
Every isomorphism class of pre-heaps contains exactly one heap and this is the unique pre-heap $E$ in the class  for which $\sum_{\alpha \in E}h(\alpha)$ is minimal. 	The pictorial representation of heaps can be seen in \cite{MR11108,viennot-imsc2,viennot-Talca,viennot-imsc5}.

Let $\mathcal{H}(I,\zeta)$ be the set of all heaps over $(I,\zeta)$. This set can be made into a monoid with a product called \emph{superposition} of heaps. To get superposition $E \circ F$ of $F$ over $E$, let the heap $F$ `fall' over $E$. Let $\mathcal{H}_{\bold k}(I,\zeta)$ be the set of all heaps of weight $\bold k$ for $\bold k \in \mathbb{Z}_{+}^{n}$ where the weight counts the number of pieces in each of the positions. This gives a $\mathbb{Z}_{+}^{n}$-gradation on $\mathcal{H}(I,\zeta)$.
We define a map $\psi : \mathcal{M}(I) \rightarrow \mathcal{H}(I,\zeta)$ as follows:  For a word $p_1\,p_2\, \cdots \,p_k \in \mathcal{M}(I)$ define $\psi(p_1\,p_2\, \cdots \,p_k) = p_1 \,\circ\, p_2 \,\circ \cdots \circ\, p_k $. Note that $\psi^{-1}(E)$ is the set of all linear orders compatible with $\le_E$. It is clear that $\psi$ extends to a weight preserving isomorphism of the monoids $\mathcal{M}(I,\eta)$ and $\mathcal{H}(I,\zeta)$.  
\subsection{Pyramids and Lyndon heaps}
 	For a heap $E$, $\min E$ is the heap composed of minimal pieces of $E$ with respect to $\le_E$ and $\max E$ is defined similarly. We write $|E|$ for the number of pieces in $E$ and $|E|_{\alpha}$ for the number of pieces of $E$ in the position $\alpha$. A heap $E$ such that $\min (E) = \{ \alpha \}$ is said to be a \emph{pyramid} with basis $\alpha$. The set of all pyramids in $\mathcal{H}(I,\zeta)$ is denoted by $\mathcal{P}(I,\zeta)$ and the set of all pyramids with basis $\alpha_i$ is denoted by $\mathcal{P}^{i}(I,\zeta)$. We note that if $E$ is a pyramid then $\pi(E)$ is a connected subset of $I$. For a heap $E$, we define $\st(E) := \max \psi^{-1}(E)$ the standard word associated to $E$. For two heaps $E,F$ we say $E \le F$ if $\st(E) \le \st(F)$. This defines a total order in the heaps monoid $\mathcal{H}(I,\zeta)$. \emph{For the rest of this paper we fix this total order in $\mathcal{H}(I,\zeta)$}.
%\end{defn}
Let $E$ be a heap, we say that $E$ is \emph{periodic} if there exists a heap $F \ne 0$ (0 - empty heap) and an integer $k \ge 2$ such that $E = F^k$.
% and \\
Similarly, $E$ is \emph{primitive} if $E = U \circ V = V \circ U$ then either $U = 0$ or $V = 0.$
%\end{defn}
%\end{defn}
%\begin{defn}
	Pyramids in which the minimum piece has the lowest position (with respect to the total order on $I$) are known as \emph{admissible pyramids}. 
	%and the set of all admissible pyramids is denoted by $\mathcal{AP}(I,\zeta)$. ie.,  $\mathcal{AP}(I,\zeta) = \{E \in \mathcal{P}(I,\zeta) : \text{min}(E) =  \text{min}~ \pi(E) ~ \}$. 
	%\begin{defn}\label{multi}
	A heap $E$ in $\mathcal{H}(I,\zeta)$ is said to be \emph{multilinear} if every basic piece occurs exactly once in $E$. The set of all multilinear heaps of $G$ is denoted by $\mathcal{H}_{\bold 1}(I,\zeta)$.  %\begin{defn}	
	An admissible pyramid which is also multilinear is known as a \emph{super letter}.
%\end{defn}
%\begin{defn}

	Let $E$ be a heap. If $E$ = $U \circ V$ for some heaps $U$ and $V$, we say that $V \circ U$ is a \emph{transpose} of $E$. The transitive closure of transposition is an equivalence relation on $\mathcal{H}(I,\zeta)$, which we call the conjugacy relation of heaps and is denoted by $\sim$.    
%\end{defn}
%\begin{defn}
	A non-empty heap $E$ is said to be \emph{Lyndon} if $E$ is primitive and minimal in its conjugacy class. We write $\mathcal{LH}(I,\zeta)$ for the set of all Lyndon heaps over $G$.
	% and $\mathcal{LH}_{\bold k}(I,\zeta)$ denotes the Lyndon heaps of weight $\bold k$. 
	It follows immediately that super letters are Lyndon. We have the following important proposition from \cite[Proposition 2.1.10]{lalonde}.
%\end{defn}
\begin{prop} \label{lyndonfact}
	Let $E \in \mathcal{H}(I,\zeta)$ then $E$ factorizes uniquely as $E = L_1 \circ L_2 \circ \cdots \circ L_k$ with $k \ge 0$, $L_i$ are Lyndon and $L_1 \ge L_2 \ge \dots \ge L_k$. This factorization is known as \emph{Lyndon factorization} and we define the Lyndon length of $E$, denoted by $ll(E)$, to be the number of factors occurring in such factorization. 
\end{prop} 
The following lemmas are immediate.
\begin{lem}
	Every Lyndon heap E is an admissible pyramid.
\end{lem}
\begin{lem}\label{iff}
	A multilinear heap $E$ is $Lyndon$ if and only if $E$ is an admissible pyramid.
\end{lem}
 \begin{rem}\label{pyrcon}
 	According to Viennot, a pyramid is a heap with a unique maximal piece \cite[Definition 5.9]{MR11108}. 
 	In this paper, we follow the Lalonde's convention on pyramids \cite[Page 173]{lalonde}, i.e., A pyramid is a heap with a unique minimal piece.
 \end{rem}

\subsection{Pyramid proportionality lemma I} \label{ppl1}
In this subsection, we shall compare the pyramids of weight $\bold k$ over $G$ and multilinear pyramids over $G(\bold k)$. First, we fix some notations.
%\begin{notation}
Let $\bold k = (k_1,k_2,\dots,k_{n})$ and  $\bold m = (m_1,m_2,\dots,m_{n}) \in \mathbb{Z}_{+}^{n}$. We say $\bold k$ is \emph{relatively prime} if its entries are.  
	Define, $\supp \bold k = \{\alpha_i \in I :k_i \ne 0\}$, $\supp_m\bold k= \{\alpha_i,\dots, \alpha_i: \alpha_i \text{ occurs $k_i$ times for }i\in I \}$ and we say that $\bold k$ is connected if $\supp \bold k$ induces a connected subgraph in $G$. We say that $\bold m \,\le\, \bold k$ if $m_i \le k_i $ for each $i \in I$. 
	% Let $\bold 1 = (1,1,\dots,1) \in \mathbb{Z}_{+}^{n}$ and let $J \subseteq I,$ we define the restriction $\bold 1|J \in \mathbb{Z}_{+}^{n}$ by $(\bold 1|J)(i) = 1$ if $i \in J$ and $0$ is $i \notin J$.
	For $\bold k = (k_1,k_2,\dots,k_n) \in \mathbb{Z}^{n}_{+}$, the symbol $\bold k(i)$ will denote the $i$-th coordinate of $\bold k$. Also, we define $\htt \bold k = \sum\limits_{i=1}^n k_i$
	and $\bold k! = \prod\limits_{i=1}^n k_i!$. 
	%\end{notation}

Fix $\bold k \in \mathbb{Z}_{+}^{n}$. 
The \emph{clan-graph} of weight $\bold k$ associated with $G$, denoted by $G(\bold k)$, is constructed as follows.  
For each $i \in I$, take a clique (complete graph) of size $\bold k(i)$ with vertex set $\{i_1,\dots, i_{\bold k(i)}\}$ and join all the vertices of $r$-th and $s$-th cliques if 
$\{r,s\}\in E(G).$ 
%\begin{rem}
$G(\bold k)$ is also known as the join of $G$ with respect to $\bold k$. Equation \eqref{connection}, gives the relation between the $\bold k$-chromatic polynomial of $G$ and chromatic polynomial of $G(\bold k)$.
%\end{rem}
The vertex set (resp. the concurrency relation) of $G(\bold k)$ is denoted by $I_{\bold k}$ (resp. $\zeta_{\bold k}$) and $\mathcal{H}(I_{\bold k},\zeta_{\bold k})$ will denote the heaps monoid over $G(\bold k)$. 

We define the indexing map $\mathcal{I} : \mathcal{P}_{\bold k}(I,\zeta)  \rightarrow \mathcal{P}_{\bold 1}(I_{\bold k},\zeta_{\bold k})$ as follows. Let $E$ be an element of $\mathcal{P}_{\bold k}(I,\zeta)$, then $E$ has $\bold k(i)$ many $\alpha_i$ for each $i \in I$. We want to label these $\alpha_i$ with the indices $\{1,2,\dots,\bold k(i)\}$ so that $\alpha_i$s will be changed to $\alpha_{i_1},\alpha_{i_2},\dots,\alpha_{i_{\bold k(i)}}$. Note that, these can be done in $\bold k(i) !$ ways and this indexing procedure will give us $\bold k ! \cdot |\mathcal{P}_{\bold k}(I,\zeta)|$ many heaps in $\mathcal{H}_{\bold 1}(I_{\bold k},\zeta_{\bold k})$. From the definition of $\zeta_{\bold k}$ (concurrency relation of the graph $G(\bold k)$), all these heaps are indeed elements of $\mathcal{P}_{\bold 1}(I_{\bold k},\zeta_{\bold k})$. This defines a $1 \rightarrow \bold k !$
association between  $\mathcal{P}_{\bold k}(I,\zeta)$ and $\mathcal{P}_{\bold 1}(I_{\bold k},\zeta_{\bold k})$ and we denote this association by $\mathcal{I}$. 
Conversely, an element $E$ of $\mathcal{P}_{\bold 1}(I_{\bold k},\zeta_{\bold k})$ has pieces  $\alpha_{i_1},\alpha_{i_2},\dots,\alpha_{i_{\bold k(i)}}$ for each $i \in I$. We forget the indices $1,2,\dots,\bold k(i)$ and rename all these pieces as $\alpha_i$ to get a heap in $\mathcal{P}_{\bold k}(I,\zeta)$. This is a $\bold k ! \rightarrow 1$ map between  $\mathcal{P}_{\bold 1}(I_{\bold k},\zeta_{\bold k})$ and $\mathcal{P}_{\bold k}(I,\zeta)$. We will denote this map by $\mathcal{F}$. This observation leads to the following lemma. %whose proof is immediate.

\begin{lem}\label{eqlem1}
	With the notations as above we have:
	\begin{equation}\label{1eqlem1}
	|\mathcal{P}_{\bold k}(I,\zeta)| = \frac{1}{\bold k !} |\mathcal{P}_{\bold 1}(I_{\bold k},\zeta_{\bold k})|.
	\end{equation}
\end{lem}
%\begin{lem}\label{eqlem2}
Further, for $i \in I$ and $1 \le q \le \bold k(i)$, we have
		\begin{equation}\label{eqlem2}
	|\mathcal{P}_{\bold k}^{i}(I,\zeta)| = \frac{\bold k(i)}{\bold k !} |\mathcal{P}_{\bold 1}^{i_q}(I_{\bold k},\zeta_{\bold k})|.
	\end{equation} 
%\end{lem}
\begin{pf}
     Equation \eqref{1eqlem1} follows from the above discussion. For proving \eqref{eqlem2},
	let $1 \le q \le \bold k(i)$ be fixed. 
	Let $E$ be an element of $\mathcal{P}_{\bold k}^{i}(I,\zeta)$, then $E$ has basis $\alpha_i$ and has $\bold k(i)$ many $\alpha_i$ in total. We want to label the $\bold k(j)$ many $\alpha_j$ in $E$ with the indices $\{1,2,\dots,\bold k(j)\}$ (for each $j \in I$) such a way that the basis $\alpha_i$ will be changed to $\alpha_{i_{q}}$. This can be done in $\frac{\bold k !}{\bold k(i)}$ many ways and hence the result.
\end{pf}

\subsection{Pyramid proportionality lemma II}

In this subsection, we will study the proportionality of the number of pyramids of weight $\bold k$ with different bases. This is our main proportionality lemma. We start with the following lemma, which is the pyramids proportionality lemma for multilinear heaps.
%\subsection{Multilinear pyramids proportionality lemma I} 

\begin{lem}\label{mainlem}
	%\begin{equation}\label{main}
	%	\sum\limits_{p \in \mathcal{P}_{\bold 1}(I,\zeta)} \frac{v(p)}{|p|}\, =  \sum\limits_{p \in \mathcal{P}_{\bold 1}^{i}(I,\zeta)} v(p)
	%\end{equation}
	For $1 \le i \le n$, we have
	\begin{equation}\label{1=n}
		|\mathcal{P}_{\bold 1}(I,\zeta)| =  n\,|\mathcal{P}_{\bold 1}^{i}(I,\zeta)|.
	\end{equation}

\end{lem}
\begin{proof}
	It is enough to prove that, 
	\begin{equation}\label{i=j}
	|\mathcal{P}_{\bold 1}^{i}(I,\zeta)| =  \,|\mathcal{P}_{\bold 1}^{j}(I,\zeta)|
	\end{equation}
	where $1 \le i \le j \le n$. Assume that $\alpha_i$ is the least element in the total order on $I$ then, by Lemma \ref{iff}, the elements of $|\mathcal{P}_{\bold 1}^{i}(I,\zeta)|$ are super letters and hence Lyndon. Now, Theorem \ref{lalonde}  shows that,
	$$|\mathcal{P}_{\bold 1}^{i}(I,\zeta)| = |\mathcal{LH}_{\bold 1}(I,\zeta)| = \dim \mathcal{L}_{\bold 1}(G) . $$
		Since $\alpha_j$ also has the privilege to be the least element in a total order on $I$, we have
	$$|\mathcal{P}_{\bold 1}^{j}(I,\zeta)| = |\mathcal{LH}_{\bold 1}(I,\zeta)| = \dim \mathcal{L}_{\bold 1}(G) . $$
\end{proof}	
\begin{lem}For $1 \le i \le n$, we have
	\begin{equation}\label{ppleq}
	|\mathcal{P}_{\bold k}(I,\zeta)| = \frac{\htt \bold k}{\bold k(i)}\,|\mathcal{P}_{\bold k}^{i}(I,\zeta)|.
	\end{equation}
 % is an element of $I$.
\end{lem}
\begin{pf}
		Applying equation \eqref{1=n} to the graph $G(\bold k)$ gives us, 
	$$|\mathcal{P}_{\bold 1}(I_{\bold k},\zeta_{\bold k})| = (\htt \bold k) |\mathcal{P}_{\bold 1}^{i_q}(I_{\bold k},\zeta_{\bold k})|$$ where $1 \le q \le \bold k(i)$. Lemma \ref{eqlem1} gives $$|\mathcal{P}_{\bold k}(I,\zeta)| = \frac{1}{\bold k !} |\mathcal{P}_{\bold 1}(I_{\bold k},\zeta_{\bold k})|$$	and $$|\mathcal{P}_{\bold k}^{i}(I,\zeta)| = \frac{\bold k(i)}{\bold k !} |\mathcal{P}_{\bold 1}^{i_q}(I_{\bold k},\zeta_{\bold k})|.$$
	Combining these equalities we get,
	%	\begin{equation}
		$$|\mathcal{P}_{\bold k}(I,\zeta)| = \frac{1}{\bold k !} |\mathcal{P}_{\bold 1}(I_{\bold k},\zeta_{\bold k})| = \frac{\htt \bold k}{\bold k !} |\mathcal{P}_{\bold 1}^{i_q}(I_{\bold k},\zeta_{\bold k})| = \frac{\htt \bold k}{\bold k(i)}\,|\mathcal{P}_{\bold k}^{i}(I,\zeta)|.$$
%	\end{equation}
This completes the proof of the lemma.
\end{pf}
We have the following important corollary.
\begin{cor}\label{i=jext} 
	For $1 \le i \le j \le n$, %\in I$,
	\begin{equation}
			|\mathcal{P}_{\bold k}^{i}(I,\zeta)| = \frac{\bold k(i)}{\bold k(j)}\,|\mathcal{P}_{\bold k}^{j}(I,\zeta)|.
	\end{equation}
\end{cor}

\subsection{Pyramid and Lyndon heap lemma}\label{free}
A pyramid $p$ with basis $\alpha$ is said to be elementary if the minimal piece of $p$ is the only one in the position $\alpha$. i.e., $|p|_{\alpha} = 1$. We write $\mathcal{P}_{e}^{i}(I,\zeta)$ for the set of elementary pyramids with basis $\alpha_i$. Let $\mathcal{P}^*_e(I,\zeta)$ be the submonoid generated by $\mathcal{P}^i_e(I,\zeta)$ in $\mathcal{H}(I,\zeta)$. We have the following proposition from \cite[Proposition 1.3.5]{lalonde}:
\begin{prop}
	$\mathcal{P}^*_e(I,\zeta)$ is freely generated by $\mathcal{P}^{i}_e(I,\zeta)$.
\end{prop}
Let $p$ be a pyramid in $\mathcal{P}^{i}_{\bold k}(I,\zeta)$ then, by the above proposition, we can write  $p = p_1 \circ p_2 \circ\ \dots \circ p_{\bold k(i)}$ where each $p_j \in \mathcal{P}_e^{i}(I,\zeta)$ and this factorization is unique.  
\begin{lem} 
	Let $p$ be a pyramid in $\mathcal{P}^{i}_{\bold k}(I,\zeta)$ with  $p = p_1 \circ p_2 \circ \dots \circ p_{\bold k(i)}$ where each $p_j \in \mathcal{P}_e^{i}(I,\zeta)$ then any cyclic rotation of $p_1 \circ p_2 \circ \dots \circ p_{\bold k(i)}$ is again an element in $\mathcal{P}^{i}_{\bold k}(I,\zeta)$. 
\end{lem}
\begin{proof}
	It is sufficient to show that  $q = p_{\bold k(i)} \circ p_1 \circ \dots \circ p_{\bold k(i)-1}$ is in $\mathcal{P}^{i}_{\bold k}(I,\zeta)$. We claim that $\min q = \{\alpha_i\}$. Since $p_{\bold k(i)} \in \mathcal{P}_e^{i}(I,\zeta)$, we have $\alpha_i \in \min q$. Let $\alpha$ be a basic piece different from $\alpha_i$ in $q$, then $\alpha \in p_j$ for some $1 \le j \le \bold k(i)$. Again, $p_j \in \mathcal{P}_e^{i}(I,\zeta)$ and hence $h(\alpha) > h(\alpha_i)$ for $\alpha_i \in \min p_j$. This shows that $\alpha \notin \min q$ and the proof is done.
\end{proof}
The following lemma is the pyramid and the Lyndon heap lemma. We note that Theorem \ref{pl} follows immediately from this lemma.
\begin{lem}
	Let $\bold k$ be an arbitrary element in $\mathbb{Z}_+^{n}$, then 
	\begin{equation}\label{pleq}
	|\mathcal{P}_{\bold k}^{j}(I,\zeta)| = \sum_{l | \bold k} \frac{\bold k(j)}{l} |\mathcal{LH}_{\frac{\bold k}{\ell}}(I,\zeta)|
	\end{equation}
		where $1 \le j \le n$.
\end{lem}
\begin{pf}
Let $\bold k$ be an arbitrary element in $\mathbb{Z}_+^{n}$ and assume that $\alpha_i$ is the least element in $I$.  We note that the elements of $\mathcal{P}^i_{\bold k}(I,\zeta)$ can be periodic. Let $q \in \mathcal{P}^i_{\bold k}(I,\zeta)$ 
%be periodic with period k, 
then, by \cite[Proposition 1.3.6]{lalonde}, there exists a unique non-periodic pyramid $p \in \mathcal{P}^i_{\frac{\bold k}{l}}(I,\zeta)$ such that $q = p^l$.
%such that $\mathcal{P} = \mathcal{P}_0^k$ where $l = \frac{m}{k}$ and $m$ is the number of pieces in %$\mathcal{P}$. 
Let $p = p_1 \circ p_2 \circ\ \dots \circ p_{\frac{\bold k}{l}(i)}$ where each $p_j \in \mathcal{P}_e^{i}(I,\zeta)$. Now, by the previous lemma, any cyclic rotation of $p_1 \circ p_2 \circ\ \dots \circ p_{\frac{\bold k}{l}(i)}$ is again an element in $\mathcal{P}^{i}_{\frac{\bold k}{l}}(I,\zeta)$. Let these cyclic rotations be denoted by $q_1, q_2, \dots, q_{\frac{\bold k}{l}(i)}$ with $q_1 = p$.
Note that, the pyramids $q_1,q_2, \dots, q_{\frac{\bold k}{l}(i)}$ are in the same conjugacy class as one is obtained from other by a sequence of transpositions of heaps. Since $p$ is not periodic, this class has a unique Lyndon heap which is also a pyramid with basis $\alpha_i$. But $q_1,q_2, \dots, q_{\frac{\bold k}{l}(i)}$ are the only pyramids with basis $\alpha_i$ in this class and hence one (and only one) among them has to be a Lyndon heap. 
Hence we have the following equation
%	\begin{equation}\label{recur}
$$	|\mathcal{P}^i_{\bold k}(I,\zeta)|=\sum\limits_{\ell|\bold k}\frac{\bold k(i)}{\ell}|\mathcal{LH}_{\frac{\bold k}{\ell}}(I,\zeta)|.$$
%	\end{equation}
Now, because of Corollary \ref{i=jext}, for an arbitrary $1 \le j \le n$ we have
%	\begin{equation}\label{pleq}
$$	|\mathcal{P}_{\bold k}^{j}(I,\zeta)| = \sum_{l | \bold k} \frac{\bold k(j)}{l} |\mathcal{LH}_{\frac{\bold k}{\ell}}(I,\zeta)|.$$
This completes the proof.
\end{pf}

\section{Proof of the main theorem}\label{heapsec}
In this section, we will prove Theorem \ref{mainthm}. Let $G$ be a simple graph with a totally ordered vertex set $I$ and edge set $E(G)$. We start with the following expression of  the $\bold k$-chromatic polynomial in terms of independent sets in $G$  \cite[Theorem 15]{MR0224505}. 
We denote by $P_k(\bold k,G)$ the set of all ordered $k$--tuples $(P_1,\dots,P_k)$ such that:
\begin{enumerate}
	\item[(i)] each $P_i$ is a non-empty independent subset of $I$, i.e., no two vertices have an edge between them; and
		\item[(ii)] the disjoint union of  $P_1,\dots,P_k$ is equal to the multiset $\{\alpha_i,\dots, \alpha_i: \alpha_i$   occurs $\bold k(i)$ times for $i\in I\}$ 
\end{enumerate} 
Then we have
\begin{equation}\label{defgenchr}\pi^G_\mathbf{k}(q)= \sum\limits_{k\ge0}|P_k(\bold k, G)| \, {q \choose k}.\end{equation}
	
	A heap $E$ is said to be \emph{trivial} if $\alpha,\beta \in \pi(E)$ implies that $\alpha \,\eta\, \beta$. i.e., $h(\alpha) = 0$ for all $\alpha \in E$. The set of all trivial heaps of $G$ is denoted by $\mathcal{TH}(I,\zeta)$. 	We assign the weights $v(\alpha) = e^{-\alpha}$ for each vertex $\alpha$ in $I$ and extend this weight to an arbitrary heap $E$ by $v(E) = \prod_{\alpha \in E}v(\alpha)$. Next, we give the statements of the fundamental lemmas of Viennot for heaps \cite{MR11108, viennot-imsc2}. We remark that the first fundamental lemma of Viennot can be viewed as the heaps theoretic analog of denominator identity for free partially commutative  Lie algebras. In \cite{viennot-imsc2}, these lemmas are proved using the method of involution on heaps.
We consider the algebra of formal power series $\mathcal{A}:=\mathbb{C}[[X_i : i\in I]]$ where $X_i = e^{-\alpha_i}$. For a formal power series $\zeta\in\mathcal{A}$ with constant term 1, its logarithm $\text{log}(\zeta)=-\sum_{k\geq 1}\frac{(1-\zeta)^k}{k}$ is well-defined. 
\medskip
	\begin{lem}\label{invlem}First fundamental lemma or Inversion lemma
			\begin{equation}
		\sum_{E \in \mathcal{H}(I,\zeta) } v(E) = \frac{1}{\sum_{E \in \mathcal{TH}(I,\zeta)}(-1)^{|E|}v(E)}.
		\end{equation}
		
\end{lem}
\medskip
\begin{lem}\label{loglem}Second fundamental lemma or Logarithmic lemma
	\begin{equation}
		\text{log}\Big(\sum_{E \in \mathcal{H}(I,\zeta) } v(E)\Big) = \sum_{p \in \mathcal{P}(I,\zeta)}\frac{v(p)}{|p|} .
		\end{equation}
\end{lem}
	\begin{rem}
		In \cite{vv}, the authors considered the logarithm of left and right-hand side of the denominator identity of Kac-Moody Lie algebras and calculated certain coefficients
		to get the expression given in Theorem \ref{kac}. The main tool was \cite[Lemma 2.3]{vv}. In our case, this lemma will be replaced by the second fundamental lemma (Lemma \ref{loglem}). By this lemma, the logarithm of the generating function for heaps gives the generating function for pyramids. So we can use the geometric properties of pyramids to prove Theorem \ref{mainthm}. 
	\end{rem}	
\subsection{Proof of \thmref{mainthm}}
We claim that
		 $$\widetilde{\pi}^G_{\bold k}(q) \,=\,\sum\limits_{\bold J \in L_G(\bold k)} \chrmult(\bold J) \,\,q^{|\bold J|} $$
	where $|\bold J|$ denotes the number of parts in $\bold J$.	
	Let 0 denote the trivial heap and let
	\begin{align*}
	N &:= 1 + \sum\limits_{\substack{E \in \mathcal{TH}(I,\zeta) \\ E \ne 0}}(-1)^{|E|} v(E)  \text{,  then} 
	\\  N^q &= \sum\limits_{k \ge 0} \binom{q}{k} \Big(\sum\limits_{\substack{E \in \mathcal{TH}(I,\zeta) \\ E \ne 0}}(-1)^{|E|} ~v(E)\Big)^k.
	\end{align*}	
	Let $\beta(\bold k) := \sum\limits_{\alpha_i \in I} k_i \alpha_i \in \mathbb{Z}[\alpha_1,\alpha_2,\dots,\alpha_n]$ then %by our choice of weights of the basic pieces $\alpha_i$s and Equation \eqref{defgenchr},  
	the coefficient of $e^{-\beta(\bold k)}$ in $N^{q}$ is equal to $(-1)^{\htt \bold k}\sum\limits_{k\ge0}|P_k(\bold k, G)| \, {q \choose k}$.
	%and $(-1)^n \sum\limits_{k = 0}^{n} \binom{q}{k}\, c_{k}(G)  $ . 
	We will calculate the same coefficient in another way. The fundamental lemmas of Viennot (Lemmas \ref{invlem} and \ref{loglem}) gives,
	 $$N^q = \text{exp}\, \Big((-q) \sum\limits_{p \in \mathcal{P}(I,\zeta)}  \frac{v(p)}{|p|}\,\Big).$$
	Expanding the right-hand side gives, 
	$$N^q = \sum\limits_{k \ge o}~ \Big(\sum\limits_{p \in \mathcal{P}(I,\zeta)} \frac{v(p)}{|p|}~\Big)^k~ \frac{(-q)^k}{k!}.$$ 
	We want to calculate the coefficient of $e^{-\beta(\bold k)}$ in the above sum, but any pyramid which has weight more than $\bold k$ won't contribute to this coefficient. Hence we can assume
	\begin{equation}\label{eqn}
		N^q = \sum\limits_{k \ge o}~ \Big(\sum\limits_{ \bold m \,\le\, \bold k}\sum\limits_{p \in \mathcal{P}_{\bold m}(I,\zeta)} \frac{v(p)}{|p|}\Big)^k~ \frac{(-q)^k}{k!}.
	\end{equation}

By Lemma \eqref{ppleq} we have,
		$$\,\,\,\,\,\,\,\,\,\, N^q= \sum\limits_{k \ge o}\, \Big(\sum\limits_{ \bold m \,\le\, \bold k }\sum\limits_{p \in \mathcal{P}_{\bold m}^{i}(J,\zeta)} \frac{v(p)}{\bold m(i)}\,\Big)^k\, \frac{(-q)^k}{k!}$$ 
	where $i$ is an element of $J := \supp \bold m$. 
By Equation \eqref{pleq} we have,
$$\,\,\,\,\,\,\,\,\,\, N^q= \sum\limits_{k \ge o}\, \Big(\sum\limits_{ \bold m \,\le\, \bold k} \sum\limits_{l | \bold {m}}\sum\limits_{p \in \mathcal{LH}_{\frac{\bold {m}}{l}}(I,\zeta)} \frac{(v(p))^l}{l}\Big)^k\, \frac{(-q)^k}{k!}.$$ 

	Consider the product $$\frac{1}{k!}\Big(\sum\limits_{ \bold m \,\le\, \bold k} \sum\limits_{l | \bold {m}}\sum\limits_{p \in \mathcal{LH}_{\frac{\bold {m}}{l}}(I,\zeta)} \frac{(v(p))^l}{l}\Big)^k $$
		$$
	=\frac{1}{k!}\sum\limits_{\substack{(\bold {m}_1, \bold {m}_2,\dots,\bold {m}_k) \\  \bold {m}_i \,\le\, \bold k}} \Big(\sum\limits_{l_1 | \bold {m}_1}\sum\limits_{p_1 \in \mathcal{LH}_{\frac{\bold {m}_1}{l_1}}(I,\zeta)}\frac{(v(p))^{l_1}}{l_1} \Big) \cdots \Big(\sum\limits_{l_k | \bold {m_k}}\sum\limits_{p_k \in \mathcal{LH}_{\frac{\bold {m}_k}{l_r}}(I,\zeta)}\frac{(v(p))^{l_k}}{l_k} \Big) $$
		$$
	=\frac{1}{k!}\sum\limits_{\substack{(\bold {m}_1, \bold {m}_2,\dots,\bold {m}_k) \\  \bold {m}_i \,\le\, \bold k }}\sum\limits_{\substack{(l_1,l_2,\dots,l_k) \\ l_i | \bold{m}_i}}\Big(\sum\limits_{p_1 \in \mathcal{LH}_{\frac{\bold {m}_1}{l_1}}(I,\zeta)}\frac{(v(p))^{l_1}}{l_1} \Big) \cdots \Big(\sum\limits_{p_k \in \mathcal{LH}_{\frac{\bold {m}_k}{l_k}}(I,\zeta)}\frac{(v(p))^{l_k}}{l_k} \Big) $$
		$$
	=\frac{1}{k!}\sum\limits_{\substack{(\bold {m}_1, \bold {m}_2,\dots,\bold {m}_k) \\  \bold {m}_i \,\le\, \bold k}}\sum\limits_{\substack{(l_1,l_2,\dots,l_k) \\ l_i | \bold{m_i}}}\sum\limits_{\substack{(p_1,p_2,\dots,p_k) \\ p_i \in \mathcal{LH}_{\frac{\bold {m}_i}{l_i}}(I,\zeta)}} \frac{(v(p_1))^{l_1}}{l_1} \frac{(v(p_2))^{l_2}}{l_2}  \cdots \frac{(v(p_k))^{l_k}}{l_k} . $$
	First, we calculate the coefficient of $e^{-\beta(\bold k)}$ in the above sum. Consider any tuple $\overline{\bold {m}}^{'} = (\bold {m}_{i_1}, \bold {m}_{i_2},\dots,\bold {m}_{i_k})$ which satisfies $\sum\limits_{j =1}^k \bold {m}_{i_j} = \bold k$. We assume that the distinct entries occuring in this $k-$tuple are $\bold {m}_1,\bold {m}_2, \dots, \bold {m}_r$, each $\bold {m}_i$ occurs $m_i$ times and
		$\sum_{i = 1}^r m_i = k$.  We note that, $\overline{\bold {m}}^{'} $ can be permuted in $\frac{k!}{m_1!  m_2! \cdots m_r!}$ many ways and all these tuples will contribute same value to the required coefficient in the above sum. 
			Now, consider another ordered $k$-tuple $\overline {\bold m} = (\bold {m}_1,\dots,\bold {m}_1,\dots,\bold {m}_r,\dots,\bold {m}_r)$ for which the minimum element of support of its entries are non-increasing, each $\bold {m}_i$ occurs $m_i$ times, $\sum\limits_{i =1}^r \sum\limits_{j=1}^{m_i} \bold {m}_i = \bold k$ and $\sum\limits_{j=1}^r m_j = k$.  
		Clearly $\overline{\bold m}$ is a permutation of the tuple $\overline{\bold {m}}^{'} $ and also represents a unique element in $L_G(\bold k)$. We consider only $\overline {\bold m}$ in place of all the permutations of $\overline{\bold {m}}^{'} $.
	 Since, we have $k!$ in the denominator of Equation \eqref{eqn},  this $\overline {\bold m}$ will contribute  $\frac{1}{m_1! \cdot m_2! \cdots m_r!}$ to the coefficient of $e^{-\beta(\bold k)}$.  
	Hence, the required coefficient of $e^{-\beta(\bold k)}$ is equal to
	$$\sum\limits_{k \ge 0}\sum\limits_{\overline{\bold m}  \,\in\, L_G(\bold k,k)}\frac{1}{m_1! \cdot m_2! \cdot \cdots m_r!} \sum\limits_{\substack{(l_1,l_2,\dots,l_k) \\ l_i | \bold{\overline{m}}(i)}}\sum\limits_{\substack{(p_1,p_2,\dots,p_k) \\ p_i \in \mathcal{LH}_{\frac{\bold {\overline{m}}(i)}{l_i}}(I,\zeta)}} \frac{1}{l_1} \frac{1}{l_2}  \cdots \frac{1}{l_k}\,(-q)^k$$
	where $L_G(\bold k,k)$ denotes the partitions in $L_G(\bold k)$ with $k$ parts and  $\overline{\bold m}(i)$ denotes the $i$th entry of $\overline{\bold m}$. 
	Now,
	 \begin{align*}
	 \sum\limits_{\substack{(l_1,l_2,\dots,l_k) \\ l_i | \bold{\overline{m}}(i)}}\sum\limits_{\substack{(p_1,p_2,\dots,p_k) \\ p_i \in \mathcal{LH}_{\frac{\bold {\overline{m}}(i)}{l_i}}(I,\zeta)}} \frac{1}{l_1} \frac{1}{l_2}  \cdots \frac{1}{l_k} 
	 &= \sum\limits_{\substack{(l_1,l_2,\dots,l_k) \\ l_i | \overline{\bold{m}}(i)}} \prod\limits_{i =1}^k\Big(\sum\limits_{p_i \in \mathcal{LH}_{\frac{\overline{\bold {m}}(i)}{l_i}}(I,\zeta)}\frac{1}{l_i}\Big)\\ 	
	 &= \sum\limits_{\substack{(l_1,l_2,\dots,l_k) \\ l_i | \overline{\bold{m}}(i)}} \prod\limits_{i =1}^k \Big(\frac{\dim (\frac{\overline{\bold {m}}(i)}{l_i})}{l_i}\Big) \\ 
	 &= \prod\limits_{i =1}^k \Big(\sum\limits_{l_i | \overline{\bold {m}}(i)}\frac{\dim (\frac{\overline{\bold {m}}(i)}{l_i})}{l_i}\Big)\\
	  &= \Big(\sum\limits_{l_i | \bold {m}_i}\frac{\dim (\frac{\bold {m}_i}{l_i})}{l_i}\Big)^{m_i}.\\
	 \end{align*}
This shows that, the coefficient of $e^{-\beta(\bold k)}$ is equal to	
$$\sum\limits_{k \ge 0}\sum\limits_{\overline{\bold m}  \,\in\, L_G(\bold k,k)} \chrmult(\overline{\bold m})(-q)^k.$$
 
	We have already shown that the coefficient of $e^{-\beta(\bold k)}$ in $N^q$ is equal to $(-1)^{\htt \bold k}\sum\limits_{k\ge1}|P_k(\bold k, G)| \, {q \choose k}$ and hence, 	
	$$(-1)^{\htt \bold k}\sum\limits_{k\ge1}|P_k(\bold k, G)| \, {q \choose k} 	=\sum\limits_{k \ge 0}\sum\limits_{\overline{\bold m}  \,\in\, L_G(\bold k,k)} \chrmult(\overline{\bold m})(-q)^k$$
		and so
	$$\widetilde{\pi}^G_{\bold k}(q) \,=\, \sum\limits_{k \ge 0}\sum\limits_{\overline{\bold m}  \,\in\, L_G(\bold k,k)} \chrmult(\overline{\bold m}) \,q^k.$$
		Equivalently,
 $$\widetilde{\pi}^G_{\bold k}(q) \,=\,\sum\limits_{\bold J \in L_G(\bold k)} \chrmult(\bold J) \,\,q^{|\bold J|}. $$	
This proves Theorem \ref{mainthm} (and hence	
 Theorem \ref{kac}) for the case of free partially commutative Lie algebras. 
	 %This also proves Corollary \ref{chmhs} using \eqref{bondhilbert}. 

We have proved the following expression of the $\bold k$-chromatic polynomial in terms of bond lattice and the root multiplicities of generalized Kac-Moody algebras (also known as Borcherds algebras) in \cite[Theorem 1]{akv}.
\begin{thm}\label{borcherds}
	Let $G$ be the graph of a Borcherds algebra $\lie g$ then, under certain assumptions on real simple roots in the support of $\bold k$, we have
	\begin{equation}\label{borcherdseq}
	\widetilde{\pi}^G_{\bold k}(q)=\sum_{\mathbf{J}\in L_{G}(\bold k)} \prod_{J\in\bold J}\binom{q\text{ mult}(\beta(J))}{D(J,\mathbf{J})}.
	\end{equation}
\end{thm}

	In Theorem \ref{mainthm} we have proved that, if $\mathcal{L}(G)$ is the free partially commutative Lie algebra  associated to the graph $G$, then
	\begin{equation}\label{heapexp1}
	\widetilde{\pi}^G_{\bold k}(q) \,=\,\sum\limits_{\bold J \in L_G(\bold k)} \chrmult \bold J \,\,q^{|\bold J|}, 
	\end{equation}
If we assume that the Borcherds algebra $\lie g$ has only imaginary simple roots, then its positive part $\lie n^+$ is isomorphic to the free partially commutative Lie algebra associated to $G$. Hence, Theorem \ref{borcherds} gives a connection between the $\bold k$-chromatic polynomials and free partially commutative Lie algebras.
In practice, Equation \eqref{borcherdseq} and \eqref{heapexp1} are equivalent. But in Equation \eqref{heapexp1},  the coefficients are given in terms of linear coefficients of $\bold m$-chromatic polynomials of smaller height and also given  in the basis $1, q, q^2, \dots, q^n, \dots$ whereas Equation \eqref{borcherdseq} is in the binomial basis of the polynomial algebra $\mathbb{C}[q]$. This is advantageous from the computational perspective. Also, Equation \eqref{heapexp1} looks more natural extension of \eqref{kaceq} over \eqref{borcherdseq}.

\section{ Proof of the corollaries}\label{acyclicsec}

\subsection{An expression for chromatic polynomials in terms of heaps}\label{bondheaps}
Let $G$ be a simple graph with the totally ordered vertex set $I$ and edge set $E(G)$. In this subsection, we prove Corollary \ref{chmhs} by studying the relation between bond lattice $L_{G}(\bold k)$ and the Hilbert series of heaps of weight $\bold k$ over $G$. We define the Hilbert series, associated with the Lyndon length statistic, of the set of heaps of weight $\bold k$ over $G$ as follows.
%	\begin{equation}
$$H_{\bold k}^G(q) = \sum\limits_{k \ge 1} | \{E \in \mathcal{H}_{\bold k}(I,\zeta) : ll(E) = k \}| q^k$$ where $ll(E)$ denotes the Lyndon length of $E$.
% \end{equation}
Given this, we claim that 	
$$\widetilde{\pi}_\bold 1^G(q) = \sum\limits_{k \ge 1} | \{E \in \mathcal{H}_{\bold 1}(I,\zeta) : ll(E) = k \}|  q^k = \sum\limits_{E \in \mathcal{H}_{\bold 1}(I,\zeta)} q^{ll(E)}.$$

Theorem \ref{mainthm} implies that $	\widetilde{\pi}_\bold 1^G(q) = \sum\limits_{\bold J \in L_G(\bold 1)} \mult \bold J\, q^{|\bold J|}$. Using this, the following proposition proves this claim. 

%The following result gives the connection between the bond lattice $L_G(\bold k)$ and the 
%Hilbert series of $\mathcal{H}_{\bold k}(I,\zeta)$. 
% from which Corollary \ref{chmhs} follows.
%We prove the following proposition, 
\begin{prop}\label{bondheap}
	With the notations as above, we have
	\begin{equation}
	H_{\bold k}^G(q) = \sum\limits_{\bold J \in L_G(\bold k)} \mult \,  \bold J \,\,q^{|\bold J|}
	\end{equation}
\end{prop} 
\begin{pf}
	Let $E \in \mathcal{H}_{\bold k}(I,\zeta)$ and let $E = p_1^{m_1} \circ p_2^{m_2} \circ \cdots \circ p_r^{m_r}$ be its Lyndon factorization, then 
	we have $ \bold J(E) := (\supp_m(p_1),\dots,\supp_m(p_1),\dots,\supp_m(p_r),\dots,\supp_m(p_r)) \in L_G(\bold k)$ where $\supp_m(p_i)$ occurs $m_i$ times and $|\bold J(E)| = ll(E)$ .
	Conversely, let	$\bold J := (J_1,\dots,J_1,J_2,\dots,J_2,\dots,J_r,\dots,J_r) \in L_G(\bold k)$ where each $J_i$ occurs $m_i$ times. Assume that the part $J_i$ has weight $\bold {m}_i$, then by Theorem \ref{lalonde} there are $\mult \bold {m}_i$ many Lyndon heaps of weight $\bold {m}_i$. 
	Consider $E = p_1^{m_1} \circ p_2^{m_2} \circ \cdots \circ p_r^{m_r} \in \mathcal{H}_{\bold k}(I,\zeta)$ where each $p_i \in \mathcal{LH}_{\bold {m}_i}(I,\zeta)$.
	We have, the minimum elements of entries of $\bold J$ are non-increasing and hence the given factorization of $E$ as the product of Lyndon heaps is indeed the Lyndon factorization. By the uniqueness of Lyndon factorization of heaps, there are $\mult \bold J = \prod\limits_{i=1}^r (\mult \bold {m}_i)^{m_i}$ many such $E$ and all these heaps have Lyndon length $|\bold J|$. This proves the result.
\end{pf}

\subsection{An expression for chromatic polynomials in terms of acyclic orientations}
In this subsection, we prove Corollary \ref{acycliccor} by studying the relation between acyclic orientations and the heaps of weight $\bold 1$ over $G$.
Let $\mathcal{O}(G)$ be the set of all acyclic orientations of $G$. A \emph{source} is a vertex with only outgoing edges and a \emph{sink} has only incoming edges.  For fixed vertex $i$, the set of all acyclic orientations in which $i$ is the only source is denoted by $\mathcal{O}^{i}(G)$. We start with the following simple lemma.
\begin{lem}\label{bijection}
	With the notations as above, we have
	$$ |\mathcal{H}_{\bold 1}(I,\zeta)| = |\mathcal{O}(G)|.$$
\end{lem}
\begin{pf}
	Let $E$ be an element of $\mathcal{H}_{\bold 1}(I,\zeta)$, we associate an acyclic orientation (of $G$) to $E$ as follows: Let $e = \{\alpha_i,\alpha_j\}$ be an arbitrary edge in $G$, then the pieces $\alpha_i$ and $\alpha_j$ are in different levels in $E$. Without loss of generality, we assume that  $\alpha_j$ stays above $\alpha_i$ in $E$, i.e., $\htt(\alpha_i) < \htt (\alpha_j)$ in $E$. In this case, orient the edge $e$ from $\alpha_i$ to $\alpha_j$ and call the resulting orientation $\mathcal{O}_E$. Since $E$ is multilinear, we get  $\mathcal{O}_E$ is acyclic. This defines a map $\mathcal{T}$ from $\mathcal{H}_{\bold 1}(I,\zeta)$ into $\mathcal{O}(G)$ and this assigning process is reversible. This shows that $\mathcal{T}$ is a bijection.
\end{pf}
An acyclic orientation $\mathcal{O}$ of $G$ is said to be Lyndon if it has a unique source at the minimum element of $I$. Let $(I_1,I_2) \in L_G(\bold 1)$ be a partition with two parts, then the subgraphs $G_1$ and $G_2$ of $G$ induced by $I_1$ and $I_2$ resp. are connected. Let $\mathcal{O}_1 \in \mathcal{O}(G_1)$ and $\mathcal{O}_2 \in \mathcal{O}(G_2)$ be acyclic orientations which are Lyndon, i.e., $\mathcal{O}_1$ and $\mathcal{O}_2$ are acyclic orientations with the unique source at the minimum elements $i$ and $j$  of $I_1$ and $I_2$ respectively. We say $\mathcal{O}_1 \ge \mathcal{O}_2$ if $i \ge j$. 
		Given two such acyclic orientations, we define their composition $\mathcal{O}_1 \circ \mathcal{O}_2$ as follows:  All the edges of $G$ which are in $G_1$ or $G_2$ will be given the orientation $\mathcal{O}_1$ or $\mathcal{O}_2$ accordingly. The edges that are straddled between $G_1$ and $G_1$ will get the orientation against the order of the minimum element of $I_1$ and $I_2$ where  we say that an edge $e$ {\em straddles} $G_1$ and $G_2$ if one end of $e$ is in $G_1$ and the other in $G_2$.  Note that $\mathcal{O}$ is indeed an acyclic orientation of $G$. 
		If $\mathcal{O} \in \mathcal{O}(G)$ then the vertex set $I$ of $G$  is said to be the support of $\mathcal{O}$. Let $(\mathcal{O}_1,\dots,\mathcal{O}_k)$ be an ordered $k$-tuple of acyclic orientations of some subgraph of $G$. Assume that each $\mathcal{O}_i$ is Lyndon and their supports form a partition in $L_G(\bold 1)$, then it is clear that the definition of composition of two acyclic orientations can be extended to $(\mathcal{O}_1,\dots,\mathcal{O}_k)$ and the resulting orientation of $G$ will be acyclic. Given this, we have the following result.% from which \eqref{acyclicexp} follows.
	\begin{prop} \label{acyclicfact}
	Let $\mathcal{O}$ be an acyclic orientation in $\mathcal{O}(G)$, then $\mathcal{O}$ factorizes uniquely as $\mathcal{O} = \mathcal{O}_1 \circ \mathcal{O}_2 \circ \cdots \circ \mathcal{O}_k$ with $k \ge 1$, $\mathcal{O}_i \in \mathcal{O}(G_i)$ (for some subgraph $G_i$ of $G$)
	% with vertex set $I_j$ and $(I_1,I_2,\dots,I_k) \in L_G(\bold 1)$) 
	are Lyndon and $\mathcal{O}_1 \ge \mathcal{O}_2 \ge \dots \ge \mathcal{O}_k.$ We call this  factorization Lyndon. We define the number of factors occurring in such factorization to be the Lyndon length of $\mathcal{O}$ and is denoted by $ll(\mathcal{O})$.
\end{prop} 
\begin{pf}
		Let $\mathcal{O}$ be an acyclic orientation in $\mathcal{O}(G)$ and let $E = \mathcal{T}^{-1}(\mathcal{O})$ be its inverse image in $\mathcal{H}_{\bold 1}(I,\zeta)$. 
First, we show that if $E$ is Lyndon then $\mathcal{O} = \mathcal{T}(E)$ is Lyndon.  By Lemma \ref{iff}, $E$ is  an admissible pyramid and hence $\mathcal{O}$ has a unique source at the minimum element of its support. Note that $\pi(E) = \supp \mathcal{O}$.  Assume that $E$ is an arbitrary element of $\mathcal{H}_{\bold 1}(I,\zeta)$ with Lyndon factorization $E = E_1 \circ E_2 \circ \cdots \circ E_k$. We will show that $\mathcal{T}(E_1) \circ \mathcal{T}(E_2) \circ \cdots \circ \mathcal{T}(E_k)$ is the Lyndon factorization of $\mathcal{O}$. From the first part of the proof, we have the acyclic orientations $\mathcal{T}(E_i)$ are Lyndon. It is clear that $\text{min}(E_1) > \text{min}(E_2) > \cdots > \text{min}(E_k)$ and hence $\mathcal{T}(E_1) \ge \mathcal{T}(E_1) \ge \dots \ge \mathcal{T}(E_k)$. This shows that $\mathcal{O} = \tau(E_1) \circ \tau(E_2) \circ \cdots \circ \tau(E_k)$ is the required Lyndon factorization. 
\end{pf}

We observe that the above-defined bijection $\mathcal{T}$ preserves the Lyndon length statistic. Now the proof for Corollary \ref{acycliccor} follows from  Corollary \ref{chmhs}. 
\begin{rem}
	In \cite[Equation 7.1]{GZ83}, we have an expression for chromatic polynomial in terms of Whitney number of poset associated with certain hyperplane arrangements. With little effort, using \cite[Theorem 7.4]{GZ83}, one can prove that Equation \eqref{1acyclicexp} and \cite[Equation (7.1)]{GZ83} are the same.
\end{rem}
%\subsection{reciprocity theorems}\label{recib}
\subsection{Recursive formula for the dimension of $\mathcal{L}_{\bold k}(G)$}
In this subsection, we prove the following corollary of Theorem \ref{pl} which gives a combinatorial formula for dimensions of grade spaces of free partially commutative Lie algebras. This result has appeared in \cite[Corollary 3.9]{akv}. We give a different proof using heaps of pieces.
\begin{cor}\label{recursionmult}
	With the notations as above, we have
	\begin{equation}\label{mult}
	\text{ mult } \beta(\bold k) = \sum\limits_{\ell | \bold k}\frac{\mu(\ell)}{\ell}\ |\pi^G_{\bold k/\ell}(q)[q]|,\end{equation}
	where $\beta(\bold k) := \sum\limits_{\alpha_i \in I} k_i \alpha_i \in \mathbb{Z}[\alpha_1,\alpha_2,\dots,\alpha_n]$ and $\mu$ is the M\"{o}bius function.
	\begin{pf}		
		We have $N = 1 + \sum\limits_{\substack{E \in \mathcal{TH}(I,\zeta) \\ E \ne 0}}(-1)^{|E|} v(E)$ which can be
		considered as an element of $\mathbb{C}[[e^{-\alpha_i} : i\in I]]$ and let $\beta(\bold k) = \sum\limits_{\alpha_i \in I} k_i \alpha_i \in \mathbb{Z}[\alpha_1,\alpha_2,\dots,\alpha_n]$. Then, the coefficient of $e^{-\beta(\bold k)}$ in $-\log N$ equals
		$$(-1)^{\mathrm{ht}(\eta(\bold k))}\sum\limits_{k\ge 1}\frac{(-1)^k}{k}|P_k(\bold k,G)|$$ which by Equation \eqref{defgenchr} is equal to $|\pi^G_{\bold k}(q)[q]|$.
		Now, using the second fundamental lemma, apply $-\text{log }$ to the left-hand side of the first fundamental lemma we get
		\begin{equation}\label{recursss} 
		|\pi^G_{\bold k}(q)[q]| = \frac{|\mathcal{P}_{\bold k}(I,\zeta)|}{\htt \bold k} = \sum_{\substack{\ell\in \mathbb{N}\\ \ell | \bold k}} \frac{1}{\ell}\text{ mult } \eta\left(\bold k/\ell\right).\end{equation} 
		where the second equality is obtained from Equation \eqref{pleq}. 	
		The statement of the corollary is now an easy consequence of the M\"{o}bius inversion formula.
	\end{pf}
\end{cor}

\section{An application - Chromatic reciprocity theorem}\label{recib}
In this section, we will prove Theorem \ref{(m,l)} by constructing combinatorial models which are counted by the derivatives of chromatic polynomials evaluated at the negative integers. Since $\chi_G^{(m)}$ is a convenient notation over $\pi_{\bold 1}^{G,(m)}$, we will denote the chromatic polynomial and its $m$-th derivative by $\chi_G(q)$ and $\chi_G^{(m)}(q)$ respectively. The coefficients of $\chi_G(q)$ are alternate in sign, this implies that the coefficients of $\chi_G^{(m)}(q)$ are also alternate in sign. We define a positive variant of $\chi_G^{(m)}(q)$  as follows: $$\widetilde{\chi}_G^{(m)}(q) = (-1)^{n-m}\,\chi_G^{(m)}(-q).$$
From Corollary \ref{heapexp} we have, 
\begin{equation}
	\widetilde{\chi}_G(q) = \sum\limits_{k \ge 1}^n \sum_{\substack{E \in \mathcal{H}_{\bold 1}(I,\zeta) \\ ll(E) = k}}  q^{k} = \sum\limits_{E \in \mathcal{H}_{\bold 1}(I,\zeta)} q^{ll(E)}.
\end{equation}
This implies that  for $0 \le m \le n$,
\begin{equation}\label{der}
	\widetilde{\chi}_G^{(m)}(q) = \sum\limits_{k \ge m}^n \sum_{\substack{E \in \mathcal{H}_{\bold 1}(I,\zeta) \\ ll(E) = k}} k (k-1) \cdots (k-m+1)  \, q^{k-m}.
\end{equation}

\begin{defn}
	Consider the set $\mathcal{LH}_{\bold 1,*}(I,\zeta)$ consists of Lyndon heaps in which each basic piece occurs at most once. Let $E$ be an element from this set and let $m$ be a positive integer.  An $m$-labelling of $E$ is defined as follows. Fix $i \in [m]$, define $f_i^E : \pi(E) \rightarrow [m]$ by $f(\alpha) = i$ for all $\alpha \in \pi(E)$ where $[m] = \{1,2,\dots,m\}$. We have $m$ many such labellings for each $E \in \mathcal{LH}_{\bold 1,*}(I,\zeta)$.  
\end{defn}

	Fix a positive integer $\lambda$. Let $E = E_1 \circ E_2 \circ \cdots \circ E_k$ be the Lyndon factorization of $E \in \mathcal{H}_{\bold 1}(I,\zeta)$. Consider the Lyndon factors $E_{1},E_{2},\dots,E_{m}$. We fix a $(k-j+1)$-labelling in $E_{j}$ for $1 \le j \le m$. 
	%This will give $k (k-1) \cdots (k-m+1)$ many labelled heaps. 
	Let $E_{m+1},E_{m+2},\dots,E_{k}$ be the remaining $k-m$ Lyndon factors of $E$. We fix a $\lambda$-labelling on each of these $k-m$ Lyndon factors.
	%this will give $\lambda^{k-m}$ labelled heaps. 
	An $(m,\lambda)$-labelling of $E$ is a function $f : I \rightarrow \mathbb{N}$ which is defined as follows. Let $\alpha \in I$ then $\alpha \in \pi(E_j)$ for exactly one $j \in [k]$, we define  $f(\alpha) = f_i^{E_j}(\alpha)$ where $f_i^{E_j}$ is the above-fixed labelling of $E_j$. This shows that, the total number of $(m,\lambda)$-labelling of $E$ is equal to $k (k-1) \cdots (k-m+1)  \, \lambda^{k-m}$. Hence the total number of $(m,\lambda)$-labelled multilinear heaps is equal to $$\sum\limits_{k \ge m}^n \sum_{\substack{E \in \mathcal{H}_{\bold 1}(I,\zeta) \\ ll(E) = k}} k (k-1) \cdots (k-m+1)  \, \lambda^{k-m}$$
	This shows that the number of $(m,\lambda)$-labelled multilinear heaps is equal to $\chi_G^{(m)}(q)$.  Lemma \ref{bijection} and Proposition \ref{acyclicfact} implies that the notion $(m,\lambda)$-labelling can be extended to the set of all acyclic orientations of $G$ and the proof of Theorem \ref{(m,l)} follows.

\begin{rem}
%	 Let $\mathcal{O}_{\lambda}(G)$ be the set of all $\lambda$-labelled acyclic orientations, then we have $$|\mathcal{O}_{\lambda}(G)| = \sum\limits_{E \in \mathcal{O}(G)} \lambda^{ll(E)} = \widetilde{\pi}^G_{\bold 1}(\lambda)$$
	
		We note that if $\mathcal{O} \in \mathcal{O}(G)$ has Lyndon factorization $\mathcal{O} = \mathcal{O}_1 \circ \cdots \circ \mathcal{O}_k$ then any edge $\{i,j\} \in E(G)$ such that $i$ in the support of $\mathcal{O}_1$ and $j$ in support of $\mathcal{O}_2$ has orientation $j$ to $i$. 
	From our definition of $\lambda$-labelling on acyclic orientations, we see that the vertices in the support of $\mathcal{O}_1$ and the vertices in the support of $\mathcal{O}_2$ receive arbitrary $m,n \in [\lambda]$ as its $\lambda$-labelling. Since there is an arrow from $j$ to $i$, we see that $\lambda$-labelled acyclic orientations are different from $\lambda$-compatible pairs.
\end{rem}

\bibliographystyle{plain}
\bibliography{Arunkumar}

\end{document}